\documentclass[11pt,twoside,leqno]{amsart}

\usepackage[utf8]{inputenc}
\usepackage[T1]{fontenc}
\usepackage[english]{babel}
\usepackage{csquotes}
\usepackage{lmodern}
\usepackage[a4paper,margin=1.25in,marginpar=1in]{geometry}
\usepackage{multirow}
\usepackage[shortlabels]{enumitem}

\usepackage{amssymb, amsmath, amsthm}
\usepackage{esint}
\usepackage{mathtools}

\usepackage[dvipsnames,svgnames,x11names,hyperref]{xcolor}      

\usepackage[pdfpagelabels]{hyperref}
\definecolor{emerlandgreen}{HTML}{2ECC71}
\definecolor{peterriverblue}{HTML}{3498DB}
\definecolor{alizarinred}{HTML}{E74C3C}
\hypersetup{%
colorlinks,
linktocpage,
urlcolor  = peterriverblue,
citecolor = emerlandgreen,
linkcolor = alizarinred}
\usepackage{listings}   
\definecolor{lstgreen}{rgb}{0,0.5,0}
\usepackage[capitalise,nameinlink]{cleveref}
\Crefname{equation}{}{}
\Crefname{conditioni}{Condition}{Conditions}
\Crefname{conditionalti}{Condition}{Conditions}
\makeatletter
\def\@endtheorem{\endtrivlist}
\makeatother
\newtheorem{theorem}{Theorem}[section]
\Crefname{theorem}{Theorem}{Theorems}
\newtheorem{lemma}[theorem]{Lemma}
\Crefname{lemma}{Lemma}{Lemmas}
\newtheorem{proposition}[theorem]{Proposition}
\Crefname{proposition}{Proposition}{Propositions}
\newtheorem{corollary}[theorem]{Corollary}
\Crefname{corollary}{Corollary}{Corollaries}

\Crefname{conjecture}{Conjecture}{Conjectures}

\Crefname{assumption}{Assumption}{Assumptions}

\theoremstyle{definition}
\newtheorem{definition}{Definition}[section]
\Crefname{definition}{Definition}{Definitions}

\Crefname{question}{Question}{Questions}

\theoremstyle{remark}
\newtheorem{remark}{Remark}
\Crefname{remark}{Remark}{Remarks}
\newtheorem{example}{Example}
\Crefname{example}{Example}{Examples}
\newtheorem*{example*}{Example}
\numberwithin{equation}{section}

\usepackage{calrsfs}
\DeclareMathAlphabet{\pazocal}{OMS}{zplm}{m}{n}

\newcommand{\dE}{\mathbb{E}}
\newcommand{\dN}{\mathbb{N}}
\newcommand{\dP}{\mathbb{P}}
\newcommand{\dR}{\mathbb{R}}
\newcommand{\dS}{\mathbb{S}}
\newcommand{\cC}{\pazocal{C}}
\newcommand{\cL}{\pazocal{L}}
\newcommand{\CE}{\mathcal{E}}
\newcommand{\dd}{\mathrm{d}}
\DeclareMathOperator{\e}{e}
\DeclareMathOperator{\Ent}{\mathbf{Ent}}
\DeclareMathOperator{\supp}{supp}

\author{Nathaël Gozlan \& Ronan Herry}
\AtEndDocument{\bigskip\footnotesize%
Nathael Gozlan\par
\textsc{MAP5 (UMR CNRS 8145), Université Paris Descartes}.
  \par
45 rue des Saints-Pères, 75270 Paris cedex 6, France.
\par
\href{mailto:natael.gozlan@parisdescartes.fr}{\texttt{natael.gozlan@parisdescartes.fr}} \par
\medskip

Ronan Herry\par
\textsc{MRU, Université du Luxembourg}.
\par
6 avenue de la Fonte, L-4363 Esch-sur-Alzette, Luxembourg.
\par
and \par
\textsc{LAMA (UMR CNRS 8050), Université Paris Est Marne la Vallée}.
\par
5 bd Descartes, 77454 Marne la Vallée Cedex 2, France.
\par
\href{mailto:ronan.herry@uni.lu}{\texttt{ronan.herry@uni.lu}}}
\title{Multiple sets exponential concentration and higher order eigenvalues}

\begin{document}

\begin{abstract}
  On a generic metric measured space, we introduce a notion of improved concentration of measure that takes into account the parallel enlargement of $k$ distinct sets.
  We show that the $k$-th eigenvalues of the metric Laplacian gives exponential improved concentration with $k$ sets.
  On compact Riemannian manifolds, this allows us to recover estimates on the eigenvalues of the Laplace-Beltrami operator in the spirit of an inequality of~\cite{ChungGrigoryanYau1996}.
\end{abstract}
\maketitle

\tableofcontents

\section*{Introduction}

Let $(M, g)$ be a smooth compact connected Riemannian manifold with its normalized volume measure $\mu$ and its geodesic distance $d$.
The Laplace-Beltrami operator $\Delta$ is then a non-positive operator whose spectrum is discrete.
Let us denote by $\lambda^{(k)}$, $k=0,1,2\ldots$, the eigenvalues of $-\Delta$ written in increasing order.
With these notations $\lambda^{(0)}=0$ (achieved for constant functions) and (by connectedness) $\lambda^{(1)}>0$ is the so-called spectral gap of $M$.

The study of the spectral gap of Riemannian manifolds is, by now, a very classical topic which has found important connections with numerous geometrical and analytical questions and properties.
The spectral gap constant $\lambda^{(1)}$ is for instance related to Poincaré type inequalities and governs the speed of convergence of the heat flow to equilibrium.
It is also related to Ricci curvature via the classical Lichnerowicz theorem~\cite{Lichnerowicz1958} and to Cheeger isoperimetric constant via Buser's theorem~\cite{Buser1982}.
We refer to~\cite{BakryGentilLedoux2014, Chavel1984} and the references therein for a complete picture.

Another important property of the spectral gap constant, first observed by Gromov and Milman~\cite{GromovMilman1983}, is that it controls exponential concentration of measure phenomenon for the reference measure $\mu$.
The result states as follows.
Define for all Borel sets $A\subset M$, its $r$-enlargement $A_{r}$ as the (open) set of all $x \in E$ such that there exists $y \in A$ with $d(x, y) < r$.
Then, for any $A \subset M$ such that $\mu(A) \geq 1/2$ it holds 
\[
\mu(A_r) \geq 1- be^{-a\sqrt{\lambda^{(1)}}r},\qquad \forall r>0,
\] 
where $a,b>0$ are some universal constants (according to~\cite[Theorem 3.1]{Ledoux2001}, one can take $b=1$ and $a=1/3$).
Note that this implication is very general and holds on any metric space supporting a Poincaré inequality (see~\cite[Corollary 3.2]{Ledoux2001}).
See also~\cite{BobkovLedoux1997, Schmuckenschlager1998, AidaStroock1994, GozlanRobertoSamson2015} for alternative derivations, generalizations or refinements of this result.

This note is devoted to a multiple sets extension of the above result.
Roughly speaking, we will see that if $A_1,\ldots,A_k$ are sets which are pairwise separated in the sense that $d(A_i,A_j):=\inf\{d(x,y) : x\in A_i, y\in A_j\}>0$ for any $i\neq j$ and $A$ is their union then the probability of $A_r$ goes exponentially fast to $1$ at a rate given by $\sqrt{\lambda^{(k)}}$ as soon as $r$ is such that the sets $A_{i,r}$, $i=1,\ldots,k$ remain separated.
More precisely, it follows from~\cref{main-result} (whose setting is actually more general) that, if $A_1,\ldots, A_k$ are such that $\mu(A_i)\geq \frac{1}{k+1}$ and $d(A_{i,r}, A_{j,r})>0$ for all $i\neq j$, then, denoting $A=A_1\cup\ldots \cup A_k$, it holds
\begin{equation}\label{eq:conc-intro}
\mu(A_r)\geq 1 - \frac{1}{k+1} \exp\left(- c \min(r^2 \lambda^{(k)} ; r\sqrt{\lambda^{(k)}})\right),
\end{equation}
for some universal constant $c$.
This kind of probability estimates first appeared, in a slightly different but essentially equivalent formulation in the work of Chung, Grigor'yan and Yau~\cite{ChungGrigoryanYau1996,ChungGrigoryanYau1997} (see also the related paper~\cite{FriedmanTillich2000} by Friedman and Tillich).
Nevertheless, the method of proof we use to arrive at~\eqref{eq:conc-intro} (based on the Courant-Fischer min-max formula for the $\lambda^{(k)}$'s) is quite different from the one of~\cite{ChungGrigoryanYau1996,ChungGrigoryanYau1997} and seems more elementary and general.
This is discussed in details in~\cref{sec:comparison}.
 
The paper is organized as follows.
In~\cref{section:concentration}, we prove~\eqref{eq:conc-intro} in an abstract metric space framework.
This framework contains, in particular, the compact Riemannian case equipped with the Laplace operator presented above.
The~\cref{sec:comparison} contains a detailed discussion of our result with the one of Chung, Grigor'yan \& Yau.
In~\cref{sec:examples}, we recall various bounds on eigenvalues on several non-negatively curved manifolds.
\cref{section:Markov} gives an extension of~\eqref{eq:conc-intro} to discrete Markov chains on graphs.
In~\cref{sec:functional}, we give a functional formulation of the results of~\cref{section:concentration,section:Markov}.
As a corollary of this functional formulation, we obtain a deviation inequality as well as an estimate for difference of two Lipschitz extensions of a Lipschitz function given on $k$ subsets.
Finally,~\cref{section:questions} discusses open questions related to this type of concentration of measure phenomenon.

\section{Multiple sets exponential concentration in abstract spaces}\label{section:concentration}
\subsection{Courant-Fischer formula and generalized eigenvalues in metric spaces}
Let us recall the classical Courant-Fischer min-max formula for the $k$-th eigenvalue ($k \in \dN$) of $-\Delta$, noted $\lambda^{(k)}$, on a compact Riemannian manifold $(M,g)$ equipped with its (normalized) volume measure $\mu$: 
\begin{equation}\label{CFmanifold}
\lambda^{(k)} = \inf_{\substack{V \subset \cC^{\infty}(M)\\\dim V = k+1}} \sup_{f \in V \setminus \{0\}} \frac{\int |\nabla f|^2\,\dd\mu}{\int f^{2}\,\dd\mu},
\end{equation}
where $\nabla f$ is the Riemannian gradient, defined through the Riemannian metric $g$ (see \textit{e.g}~\cite{Chavel1984}) and $|\nabla f|^{2} = g(\nabla f, \nabla f)$.
The formula~\eqref{CFmanifold} above does not make explicitly reference to the differential operator $\Delta$.
It can be therefore easily generalized to a more abstract setting, as we shall see below.

In all what follows, $(E,d)$ is a complete, separable metric space and $\mu$ a reference Borel probability measure on $E$.
Following~\cite{CheegerLipschitz}, for any function $f \colon E \to \dR$ and $x \in E$, we denote by $|\nabla f| (x)$ the \emph{local Lipschitz constant} of $f$ at $x$, defined by
\[
  |\nabla f|(x) = 
  \begin{cases}
    0 \ \text{if $x$ is isolated}\\
    \limsup_{y \to x} \frac{|f(x)-f(y)|}{d(x,y)} \ \text{otherwise}.
  \end{cases}
\]
Note that when $E$ is a smooth Riemannian manifold, equipped with its geodesic distance $d$, then, the local Lipschitz constant of a differentiable function $f$ at $x$ coincides with the norm of $\nabla f(x)$ in the tangent space $T_{x}E$.
With this notion in hand, a natural generalization of~\eqref{CFmanifold} is as follows (we follow~\cite[Definition 3.1]{Milman2015}):
\begin{equation}\label{metric_eigenvalues}
\lambda^{(k)}_{d,\mu} := \inf_{\substack{V \subset H^1(\mu)\\\dim V = k+1}} \sup_{f \in V \setminus \{0\}} \frac{\int |\nabla f|^2\,\dd\mu}{\int f^{2}\,\dd\mu}, \qquad k\geq 0,
\end{equation}
where $H^1(\mu)$ denotes the space of functions $f \in L^2(\mu)$ such that $\int |\nabla f|^2\,\dd\mu <+\infty$.
In order to avoid heavy notations, we drop the subscript and we simply write $\lambda^{(k)}$ instead of $\lambda^{(k)}_{d,\mu}$ within this section.

\subsection{Statement of the main results}
To state our first main result, we need further notations: for any $k\geq 1$, we denote by $\Delta_{k}$ the set of vectors $(a_{1},\ldots,a_{k}) \in {[0,1]}^{k}$ satisfying the following linear constraints
\[
\sum_{j=1}^{k} a_j \leq 1 \qquad \text{and} \qquad a_i + \sum_{j=1}^{k} a_{j} \geq 1, \, \forall i \in \{1,\ldots, k\}.
\]
Recall the classical notation $d(A,B) = \inf \{d(x,y) : x \in A, y\in B\}$ of the distance between two sets $A,B \subset E$.

The following theorem is the main result of the paper and is proved in~\cref{section:proofs}.

\begin{theorem}\label{main-result}
There exists a universal constant $c>0$ such that, for any $k\geq 1$ and for all sets $A_1,\ldots,A_k \subset E$ such that $\min_{i\neq j} d(A_i,A_j)>0$ and $(\mu(A_1),\ldots,\mu(A_k)) \in \Delta_k$, the set $A=A_1\cup A_2\cup \cdots \cup A_k$ satisfies
\[
\mu(A_r) \geq 1 - (1-\mu(A)) \exp\left( - c \min(r^2 \lambda^{(k)} ; r\sqrt{\lambda^{(k)}})   \right),
\]
for all $0< r \leq \frac{1}{2}\min_{i\neq j} d(A_i,A_j)$, where $\lambda^{(k)}\geq0$ is defined by~\eqref{metric_eigenvalues}.
\end{theorem}
Note that, since $(1/(k+1),\ldots, 1/(k+1)) \in \Delta_k$,~\cref{main-result} immediately implies Inequality~\eqref{eq:conc-intro}.

Inverting our concentration estimate, we obtain the following statement that provides a bound on the $\lambda^{(k)}$'s.

\begin{proposition}\label{proposition:bound-eigenvalue}
  Let $(E,d,\mu)$ be a metric measured space and $\lambda^{(k)}$ be defined as in~\eqref{metric_eigenvalues}.
  Let $A_{1}, \dots, A_{k}$ be measurable sets such that $(\mu(A_{1}), \dots, \mu(A_{k})) \in \Delta_{k}$, then, with $r = \frac{1}{2} \min_{i \ne j} d(A_{i},A_{j})$ and $A_{0} = E \setminus {(\cup A_{i})}_{r}$,
  \[
    \lambda^{(k)} \leq \frac{1}{r^{2}} \psi\left(\frac{1}{c}  \min_{i} \ln \frac{\mu(A_{i})}{\mu(A_{0})} \right),
  \]
  where $\psi(x) = \max(x,x^{2})$.
\end{proposition}

\begin{proof}
 Let $A = \cup_{i} A_{i}$.
  Inverting the formula in~\cref{main-result}, we obtain
  \[
    \lambda^{(k)} \leq \frac{1}{r^{2}} \psi\left(\frac{1}{c} \ln \frac{1-\mu(A)}{1-\mu(A_{r})}\right),
   \]
   where $\psi(x) = \max(x,x^{2})$.
  By definition of $\Delta_{k}$,
  \[
    1 - \mu(A) = 1 - \sum_{i} \mu(A_{i}) \leq \min_{i} \mu(A_{i}).
  \]
  Therefore, letting $A_{0} = E \setminus A_{r}$, we obtain the announced inequality by non-decreasing monotonicity of $\psi$ and $\ln$.
\end{proof}

The collection of sets $\Delta_k$, $k\geq 1$ has the following useful stability property:
\begin{lemma}\label{Delta}
Let $I_1, I_2, \ldots, I_n$ be a partition of $\{1,\ldots,k\}$, $k\geq 1$.
Let $a=(a_1,\ldots,a_k) \in \mathbb{R}^k$ and define $b=(b_1,\ldots,b_n)\in \mathbb{R}^n$ by setting $b_i = \sum_{j\in I_i}a_j$, $i\in \{1,\ldots,n\}$.
If $a \in \Delta_k$ then $b \in \Delta_n$.
\end{lemma}
\begin{proof}
The proof is obvious and left to the reader.
\end{proof}
Thanks to this lemma it is possible to iterate~\cref{main-result} and to obtain a general bound for $\mu(A_r)$ for all values of $r>0$.
This bound will depend on the way the sets $A_{1,r},\ldots, A_{k,r}$ coalesce as $r$ increases.
This is made precise in the following definition.

\begin{definition}[Coalescence graph of a family of sets]
Let $A_1,\ldots,A_k$ be subsets of $E$.
The \emph{coalescence graph} of this family of sets is the family of graphs $G_r=(V,E_r)$, $r>0$, where $V = \{1,2,\ldots,k\}$ and the set of edges $E_r$ is defined as follows: $\{i,j\} \in E_r$ if $d(A_{i,r},A_{j,r})=0$.
\end{definition}

\begin{corollary}\label{cor:iteration}
Let $A_1,\ldots,A_k$ be subsets of $E$ such that $\min_{i\neq j} d(A_i,A_j)>0$ and $(\mu(A_1),\ldots, \mu(A_k)) \in \Delta_k$.
For any $r>0$, let $N(r)$ be the number of connected components in the coalescence graph $G_r$ associated to $A_1,\ldots,A_k$.
The function $(0,\infty) \to \{1,\ldots,k\} :  r \mapsto N(r)$ is non-increasing and right-continuous.
Define $r_i = \sup\{ r >0 : N(r)\geq k-i+1\}$, $i=1,\ldots, k$ and $r_0=0$ then it holds
\begin{equation}\label{eq:iteration}
\mu(A_{r}) \geq 1 - (1-\mu(A)) \exp \left(-  c\sum_{i=1}^{k} \phi \left({[r \wedge r_{i} - r_{i-1}]}_{+} \sqrt{\lambda^{(k-i+1)}}\right)\right),\, \forall r>0,
\end{equation}
where $\phi(x) = \min(x ; x^2)$, $x\geq0$ and $c$ is the universal constant appearing in~\cref{main-result}.
\end{corollary}
Observe that, contrary to usual concentration results, the bound given above depends on the geometry of the set $A$.\\


\subsection{Proofs}
\label{section:proofs}
First, we prove~\cref{cor:iteration}.
The main argument is to repeatedly apply~\cref{main-result} until two sets or more coalesce.
\begin{proof}[Proof of~\cref{cor:iteration}]
We proceed by induction over the number of components $k$.
For $k=1$,~\eqref{eq:iteration} follows immediately from~\cref{main-result}.
Let $k>1$ and let us assume that~\eqref{eq:iteration} is true for any collection of subsets $B_1,\ldots,B_l$ satisfying the assumptions of~\cref{cor:iteration} for all $l\in \{1,\ldots, k-1\}$.
Let $A_1,A_2,\ldots,A_k$ be a collection of sets satisfying the assumptions of~\cref{cor:iteration}.
According to~\cref{main-result}, it holds
\[
\mu(A_r) \geq 1 - (1-\mu(A)) \exp\left( - c \phi(r\sqrt{\lambda^{(k)}})   \right),
\]
for all $0<r\leq \frac{1}{2}\min_{i\neq j} d(A_i,A_j)$.

Let $k_{1} = N(\frac{1}{2} \min_{i\neq j} d(A_{i}, A_{j}))$ and let $i_{1} = k - k_{1}$.
Then, for all $i\in \{1,\ldots,i_{1}\}$, $r_{i} = \frac{1}{2} \min_{i \neq j} d(A_{i},A_{j})$.
So that, for all $0 < r \leq r_{i_{1}}$, the preceding bound can be rewritten as follows (note that only the term of index $i=1$ gives a non zero contribution)
\begin{equation}\begin{split}
  \label{eq:iteration-proof}
  \mu(A_r) &\geq 1 - (1-\mu(A)) \exp \left( - c \sum_{i=1}^{i_{1}} \phi \left( {[r\wedge r_{i} - r_{i-1}]}_{+} \sqrt{\lambda^{(k-i+1)}} \right)   \right)\\
           & = 1 - (1-\mu(A)) \exp \left( - c \sum_{i=1}^{k}\phi\left( {[r \wedge r_i - r_{i-1}]}_{+} \sqrt{\lambda^{(k-i+1)}} \right) \right)
\end{split}
\end{equation}
which shows that~\eqref{eq:iteration} is true for $0 < r \leq r_{{i}_1}$.
Now let $I_{1},\ldots, I_{k_{1}}$ be the connected components of $G_{r_{1}}$ and define, for all $i\in \{1,\ldots,k_{1}\}$, $B_{i} = \cup_{j \in I_{i}} A_{j,r_{1}}$.
It follows easily from~\cref{Delta} that $(\mu(B_1),\ldots,\mu(B_{k_1})) \in \Delta_{k_1}$.
Since $\min_{i\neq j} d(B_i,B_j)>0$, the induction hypothesis implies that 
\[
\mu(B_s) \geq 1 - (1-\mu(B)) \exp \left( -c \sum_{i=1}^{k_1} \phi \left( {[s\wedge s_i - s_{i-1}]}_{+} \sqrt{\lambda^{(k_{1}-i+1)}} \right) \right),\qquad \forall s>0,
\]
where $B = B_1\cup \cdots \cup B_{k_1} = A_{r_1}$ and $s_i=\sup\{ s >0 : N'(s)\geq k_1-i+1\}$, $i\in \{1,\ldots, k_1\}$ ($s_0=0$) with $N'(s)$ the number of connected components in the graph $G'_s$ associated to $B_1,\ldots,B_{k_1}$.
It is easily seen that $r_{i_1+i} = r_{i_1}+ s_i$, for all $i\in \{0,1\ldots, k_1\}$.
Therefore, we have that, for $r>r_{i_1}$,
\begin{align*}
\mu(A_r) & \geq \mu(B_{r-r_{i_1}})\\
 & \geq 1 - (1-\mu(A_{r_{i_{1}}})) \exp\left(- c \sum_{i = i_{1} +1}^{k} \phi \left({[r\wedge r_{i}-r_{i-1}]}_{+} \sqrt{\lambda^{(k-i+1)}} \right) \right)\\
 & \geq 1 - (1-\mu(A)) \exp \left(- c \sum_{i=1}^{k} \phi\left({[r\wedge r_{i}-r_{i-1}]}_{+} \sqrt{\lambda^{(k-i+1)}} \right) \right),
\end{align*}
where the last line is true by~\eqref{eq:iteration-proof}.
\end{proof}

To prove~\cref{main-result}, we need some preparatory lemmas.
Given a subset $A \subset E$, and $x \in E$, the minimal distance from $x$ to $A$ is denoted by
\[
d(x,A) = \inf_{y\in A} d(x,y).
\]

\begin{lemma}\label{lemma:inclusion}
Let $A \subset E$ and $\epsilon > 0$, then ${(E \setminus A_{\epsilon})}_{\epsilon} \subset E \setminus A$.
\end{lemma}

\begin{proof}
Let $x \in {(E \setminus A_{\epsilon})}_{\epsilon}$.
Then, there exists $y \in E \setminus A_{\epsilon}$ (in particular $d(y, A) \geq \epsilon$) such that $d(x, y) < \epsilon$.
Since the function $z\mapsto d(z,A)$ is $1$-Lipschitz, one has
  \[
    d(x, A) \geq d(y, A) - d(x, y) > 0
  \]
and so $x \in E\setminus A$.
\end{proof}

\begin{remark}
  In fact, we proved that ${(E \setminus A_{\epsilon})}_{\epsilon} \subset E \setminus \bar{A}$.
    The converse is, in general, not true.
\end{remark}

\begin{lemma}\label{lemma:max_of_quotient}
Let $A_{1}, \dots, A_{k}$ be a family of sets such that $(\mu(A_1),\ldots,\mu(A_k))\in \Delta_k$ and $r:=\frac{1}{2}\min_{i\neq j}d(A_{i}, A_{j})>0$.
Let $0<\epsilon \leq r$ and set $A = \cup_{1\leq i \leq k} A_{i}$ and $A_{0} = E \setminus (A_{\epsilon})$.
Then,
\begin{equation}  \label{ineq:bound_max}
\max_{i=0, \dots, k} \frac{\mu(A_{i,\epsilon})}{\mu(A_{i})} \leq \frac{1 - \mu(A)}{1 -\mu(A_{\epsilon})}.
\end{equation}
\end{lemma}
  
\begin{proof}
  First, this is true for $i=0$.
  Indeed, by definition $A_0 = E \setminus (A_{\epsilon})$ and, according to~\cref{lemma:inclusion}, ${(A_0)}_\epsilon \subset A^c$ (the equality is not always true), which proves~\eqref{ineq:bound_max} in this case.
  Now, let us show~\eqref{ineq:bound_max} for the other values of $i$.
  Since $\epsilon \leq r$, the $A_{j,\epsilon}$'s are disjoint sets.
  Thence,~\cref{ineq:bound_max} is equivalent to 
  \[
    \left(1-\sum_{j=1}^{k} \mu(A_{j,\epsilon})\right) \mu(A_{i,\epsilon}) \leq \left(1-\sum_{j=1}^{k} \mu(A_{j})\right) \mu(A_{i}).
  \]
  This inequality is true as soon as 
  \[
    \left(1- \mu(A_{i,\epsilon}) - m_i\right) \mu(A_{i,\epsilon}) \leq \left(1-\mu(A_i)-m_i\right) \mu(A_i),
  \]
  denoting $m_i= \sum_{j\neq i}^k \mu(A_j)$.
  The function $f_i(u) = (1-u-m_i)u$, $u\in [0,1]$, is decreasing on the interval $[(1-m_i)/2,1]$.
  We conclude from this that~\eqref{ineq:bound_max} is true for all $i \in \{1,\ldots,k\}$, as soon as $\mu(A_i) \geq (1-m_i)/2$ for all $i \in \{1,\ldots, k\}$ which amounts to $(\mu(A_1),\ldots,\mu(A_k))\in \Delta_k$.
\end{proof}

For $p>1$, we define the function $\chi_{p} \colon [0,\infty[ \to [0,1]$ by
\[
\chi_p(x) = {(1-x^{p})}^{p},\quad \text{for}\quad x \in [0,1] \quad\text{and}\quad \chi_p(x)=0 \quad\text{for}\quad x>1.
\]
It is easily seen that $\chi_p(0)=1$, $\chi_p'(0)=\chi_p(1)=\chi_p'(1)=0$, that $\chi_{p}$ takes values in $[0,1]$ and that $\chi_{p}$ is continuously differentiable on $[0,\infty[$.
We use the function $\chi_{p}$ to construct smooth approximations of indicator functions on $E$, as explained in the next statement.
\begin{lemma}\label{lemma:smooth_indicator}
Let $A \subset E$ and consider the function $f(x) = \chi_p(d(x,A)/\epsilon)$, $x \in E$, where $\epsilon>0$ and $p>1$.
For all $x \in E$, it holds
\[
|\nabla f|(x) \leq p^2\epsilon^{-1} \mathbf{1}_{A_\epsilon\setminus A}
\]
\end{lemma}

\begin{proof}
  Thanks to the chain rule for the local Lipschitz constant (see e.g.~\cite[Proposition 2.1]{AmbrosioGhezziSobolevBVMMS}),
\[ 
\left|\nabla \chi_p\left(\frac{d(\cdot, A)}{\epsilon}\right)\right|(x) \leq \epsilon^{-1} \chi_p'\left(\frac{d(\cdot, A)}{\epsilon}\right) |\nabla d(\cdot, A)|(x).
\]
The function $d(\cdot,A)$ being Lipschitz, its local Lipschitz constant is $\leq 1$ and, thereby,
\[
|\nabla f|(x) \leq \chi_p'\left(\frac{d(x,A)}{\epsilon}\right).
\]
In particular, thanks to the aforementioned properties of $\chi$, $|\nabla f|$ vanishes on $A$ (and even on $\overline{A}$) and on $\{x \in E : d(x,A)\geq \epsilon\} = E \setminus A_\epsilon$.
On the other hand, a simple calculation shows that $|\chi_p'| \leq p^2$ which proves the claim.
\end{proof}

\begin{proof}[Proof of~\cref{main-result}]
Take Borel sets $A_{1}, \dots, A_{k}$ with $\frac{1}{2}\min_{i\neq j}d(A_{i}, A_{j})\geq r>0$ and $(\mu(A_1),\ldots,\mu(A_k))\in \Delta_k$ and consider $A=A_1\cup \cdots \cup A_k$.
Let us show that, for any $0<\epsilon\leq r$, it holds
\begin{equation}\label{eq:mr1}
\left(1 + \lambda^{(k)} \epsilon^{2}\right) (1 - \mu(A_{\epsilon})) \leq (1 - \mu(A)).
\end{equation}
Let $A_0 = E\setminus (A_\epsilon)$ and set $f_i(x)= \chi_p(d(x,A_i)/\epsilon)$, $x\in E$, $i\in \{0,\ldots,k\}$, where $p>1$.
According to~\cref{lemma:smooth_indicator} and the fact that $f_i = 1$ on $A_{i}$, we obtain
\begin{equation}\label{eq:maj}
\int |\nabla f_{i}|^{2} \,\dd \mu = \frac{p^4}{\epsilon^{2}} \mu(A_{i,\epsilon} \setminus A_{i})\quad\text{and}\quad \int f_{i}^{2} \,\dd \mu \geq \mu(A_{i}).
\end{equation}
Since the $f_{i}$'s have disjoint supports they are orthogonal in $L^2(\mu)$ and, in particular, they span a $k+1$ dimensional subspace of $H^1(\mu)$.
Thus, by definition of $\lambda^{(k)}$,
\[
\lambda^{(k)} \leq \sup_{a \in \dR^{k+1}}
\frac{ \int {|\nabla \left(\sum_{i=0}^{k} a_{i} f_{i} \right)|}^{2} \,\dd \mu}{ \int {\left(\sum_{i=0}^{k} a_{i} f_{i} \right)}^{2} \,\dd \mu}
\leq \sup_{a \in \dR^{k+1}}
\frac{ \int {\left(\sum_{i=0}^{k}|a_i||\nabla f_i|\right)}^{2}\, \dd \mu}{ \int {\left( \sum_{i=0}^{k} a_{i} f_{i} \right)}^2 \,\dd \mu},
\]
where the second inequality comes from the following easy to check sub-linearity property of the local Lipschitz constant:
\[
|\nabla \left(a f + bg\right)| \leq |a||\nabla f| + |b||\nabla g|.
\]
Since the $f_i's$ and the $|\nabla f_i|'s$ are two orthogonal families, we conclude using~\eqref{eq:maj}, that 
\[
\frac{\lambda^{(k)}\epsilon^{2}}{p^4} \leq \sup_{a\in \mathbb{R}^{k+1}} \frac{\sum_{i=0}^{k} a_{i}^{2} \left(\mu(A_{i,\epsilon}) - \mu(A_{i}) \right)}{\sum_{i=0}^{k} a_{i}^{2} \mu(A_{i})},
\]
which amounts to
\begin{equation}\label{equation:max}
1 + \frac{\lambda^{(k)} \epsilon^{2}}{p^4} \leq \max_{i=0, \dots, k} \frac{\mu(A_{i,\epsilon})}{\mu(A_{i})}.
\end{equation}
Applying~\cref{lemma:max_of_quotient} and sending $p$ to $1$ gives~\eqref{eq:mr1}.
Now, if $n \in \dN$ and $0<\epsilon$ are such that $n\epsilon \leq r$, then iterating~\eqref{eq:mr1} immediately gives
\[
{\left(1 + \lambda^{(k)} \epsilon^{2}\right)}^{n} (1 - \mu(A_{n \epsilon})) \leq 1-\mu(A).
\]
Optimizing this bound over $n$ for a fixed $\varepsilon$ gives
\[
(1-\mu(A_r)) \leq (1-\mu(A)) \exp\left( - \sup\left\{ \lfloor r/\epsilon \rfloor  \log \left(1+ \lambda^{(k)}\epsilon^2\right) : \epsilon\leq r\right\}\right).
\]
Thus, letting 
\begin{equation}\label{eq:Psi}
\Psi(x) = \sup\left\{ \lfloor t \rfloor  \log \left(1+\frac{x}{t^2}\right) : t\geq 1\right\},\qquad x\geq0,
\end{equation}
it holds
\[
(1-\mu(A_r)) \leq (1-\mu(A))\exp\left(-\Psi\left(\lambda^{(k)}r^2\right)\right).
\]
Using~\cref{technique} below, we deduce that $\Psi\left(\lambda^{(k)}r^2\right) \geq c \min (r^2\lambda^{(k)} ; r\sqrt{\lambda^{(k)}})$, with $c = \log(5)/4$, which completes the proof.
\end{proof}
\begin{lemma}\label{technique}
The function $\Psi$ defined by~\eqref{eq:Psi} satisfies
\[
\Psi(x) \geq \frac{\log(5)}{4} \min(x ; \sqrt{x}),\qquad \forall x \geq0.
\]
\end{lemma}
\begin{proof}
Taking $t=1$, one concludes that $\Psi(x) \geq \log(1+x)$, for all $x\geq0$.
The function $x\mapsto \log(1+x)$ being concave, the function $x\mapsto \frac{\log(1+x)}{x}$ is non-increasing.
Therefore, $\log(1+x) \geq \frac{\log(5)}{4} x$ for all $x\in [0,4]$.
Now, let us consider the case where $x\geq 4$.
Observe that $\lfloor t \rfloor \geq t/2$ for all $t\geq 1$ and so, for $x\geq 4$,
\[
\Psi(x) \geq \frac{1}{2}\sup_{t\geq 1}\left\{ t \log\left(1+\frac{x}{t^2}\right)\right\} \geq \frac{\log (5)}{4} \sqrt{x},  
\]
by choosing $t = \sqrt{x}/2 \geq 1$.
Thereby, 
\[
\Psi(x) \geq \frac{\log(5)}{4} \left[x \mathbf{1}_{[0,4]}(x) + \sqrt{x} \mathbf{1}_{[4,\infty)}(x)\right] \geq \frac{\log(5)}{4} \min(x ; \sqrt{x}),
\]
which completes the proof.
\end{proof}

\begin{remark}\label{rem:c}
The conclusion of Lemma~\cref{technique} can be improved.
Namely, it can be shown that 
\[
\Psi(x) = \max\left( 
\left(1+\lfloor \frac{\sqrt{x}}{a} \rfloor \right)
\log
\left(
1 + \frac{x}{{\left(1+\lfloor\frac{\sqrt{x}}{a} \rfloor \right)}^{2}}
\right) \ ;\ \left(\lfloor \frac{\sqrt{x}}{a} \rfloor \right)
\log
\left(
1 + \frac{x}{{\left(\lfloor\frac{\sqrt{x}}{a} \rfloor \right)}^{2}}
\right)
\right),
\]
(the second term in the maximum being treated as $0$ when $\sqrt{x}<a$)
where $0<a<2$ is the unique point where the function $(0,\infty) \to \mathbb{R} : u \mapsto \log(1+u^2)/u$ achieves its supremum.
Therefore, 
\[
\Psi(x) \sim \frac{\log(1+a^2)}{a} \sqrt{x}
\] 
when $x \to \infty$.
The reader can easily check that $\frac{\log(1+a^2)}{a}\simeq 0.8$.
In particular, it does not seem possible to reach the constant $c=1$ in~\cref{main-result} using this method of proof.
\end{remark}

\subsection{Two more multi-set concentration bounds}
The condition $(\mu(A_{1}), \dots, \mu(A_{k})) \in \Delta_{k}$ can be seen as the multi-set generalization of the condition, standard in concentration of measure, that the size of the enlarged set has to be bigger than $1/2$.
Indeed, the reader can easily verify that $(\frac{1}{k+1},\dots,\frac{1}{k+1}) \in \Delta_{k}$.
However, in practice, this condition can be difficult to check.
We provide two more multi-set concentration inequalities that hold in full generality.
The method of proof is the same as for~\cref{main-result} and is based on~\cref{equation:max}.

\begin{proposition}\label{proposition:alternative-result}
  Let $(E,d,\mu)$ be a metric measured space and $\lambda^{(k)}$ be defined as in~\eqref{metric_eigenvalues}.
  Let $(A_{1}, \dots, A_{k})$ be $k$ Borel sets, $A = \cup_{i} A_{i}$ and $A_{0} = E \setminus A_{r}$.
  Then, with $a_{(1)} = \min_{1 \leq i \leq k} \mu(A_{i})$, the following two bounds hold:
  \begin{align*}
  1 - \mu(A_{r}) & \leq (1-\mu(A)) \frac{1}{\prod_{i=1}^{k} \mu(A_{i})} \exp{\left(-c \min\left(r^{2} \lambda^{(k)}, r \sqrt{\lambda^{(k)}}\right)\right)};\\
    1 - \mu(A_{r}) & \leq (1-\mu(A)) \frac{1}{{\mu(A)}^{\mu(A)/a_{(1)}}} \exp{\left(-c \min\left(r^{2} \lambda^{(k)}, r \sqrt{\lambda^{(k)}}\right)\right)}.
  \end{align*}
\end{proposition}

\begin{proof}
  Fix $N \in \dN$ and $\epsilon > 0$ such that $N \epsilon \leq r$.
  For $i =1,\dots,k$ and $n \leq N$, we define
  \begin{align*}
    \alpha_{i}(n) &= \frac{\mu(A_{i,n\epsilon})}{\mu(A_{i,(n-1)\epsilon})};\\
    M_{n} &= \max_{1 \leq i \leq k} \alpha_{i}(n) \vee \frac{1 - \mu(A_{(n-1)\epsilon})}{1 - \mu(A_{n\epsilon})};\\
    L_{n} &= \{ i \in \{1,\dots,k\} | M_{n} = \alpha_{i}(n) \};\\
    N_{i} &= \sharp \{ n \in \{1,\dots,N\} | i = \inf L_{n} \};\\
    N_{0} &= N - \sum_{i=1}^{k} N_{i}.
\end{align*}
Roughly speaking, the number $N_{i}$ ($0 \leq i \leq k$) counts the number of time where the set $A_{i}$ growths in iterating~\cref{equation:max}.
\cref{lemma:max_of_quotient} asserts that in the case where $(\mu(A_{1}),\dots,\mu(A_{k})) \in \Delta_{k}$, then $N_{0} = N$.
However, we still obtain from~\cref{equation:max}, for $1 \leq i \leq k$,
\begin{equation}\label{equation:bound_Ni}
  \frac{1}{\mu(A_{i})} \geq \prod_{n=1}^{N} \alpha_{i}(n) \geq  {\left(1 + \lambda^{(k)} \epsilon^{2}\right)}^{N_{i}}.
\end{equation}
The first inequality is true because $\mu(A_{i,N\epsilon}) \leq 1$ and a telescoping argument.
The second inequality is true because, as $n$ ranges from $1$ to $N$, by definition of the number $N_{i}$ and~\cref{equation:max}, there are, at least $N_{i}$ terms appearing in the product that can be bounded by $(1+\lambda^{(k)}\epsilon^{2})$.
The other terms are bounded above by $1$.
The case of $i=0$ is handled in a similar fashion and we obtain:
\begin{equation}\label{equation:bound_N0}
  \begin{split}
    1 - \mu(A_{N\epsilon}) &\leq (1 - \mu(A)) {\left(1 + \lambda^{(k)} \epsilon^{2}\right)}^{-N_{0}}\\
                           &= (1 - \mu(A)) {\left(1 + \lambda^{(k)} \epsilon^{2}\right)}^{-N} \prod_{i=1}^{k} {\left(1 + \lambda^{(k)} \epsilon^{2}\right)}^{N_{i}}.
  \end{split}
\end{equation}
The announced bounds will be obtain by bounding the product appearing in the right-hand side and an argument similar to the end of the proof of~\cref{main-result}.
From~\cref{equation:bound_Ni}, we have that,
\begin{equation}\label{equation:bound_product1}
  \prod_{i=1}^{k} {\left(1+\lambda^{(k)}\epsilon^{2}\right)}^{N_{i}} \leq \frac{1}{\prod_{i=1}^{k} \mu(A_{i})}.
\end{equation}
Also, from~\cref{equation:bound_Ni},
\[
  \mu(A_{i,N\epsilon}) \geq {\left(1+\lambda^{(k)}\epsilon^{2}\right)}^{N_{i}}\mu(A_{i}).
\]
Because $N\epsilon \leq r$, the sets $A_{1,N\epsilon}, \dots, A_{k,N\epsilon}$ are pairwise disjoint and, thereby,
\[
  1 \geq \sum \mu(A_{i,N\epsilon}) \geq \sum_{i=1}^{k} {\left(1+\lambda^{(k)}\epsilon^{2}\right)}^{N_{i}} \mu(A_{i}).
\]
Fix $\theta > 0$ to be chosen later.
By convexity of $\exp$,
\begin{align*}
  1 + (1-\mu(A)) {\left(1+\lambda^{(k)}\epsilon^{2}\right)}^{\theta} & \geq \exp{\left(  \left(\sum_{i=1}^{k} \mu(A_{i}) N_{i} + (1-\mu(A))\theta\right) \log \left(1+\lambda^{(k)}\epsilon^{2}\right) \right)}\\
                                                        & \geq \exp{\left(  \left(a_{(1)} \sum_{i=1}^{k} N_{i} + (1-\mu(A))\theta\right) \log \left(1+\lambda^{(k)}\epsilon^{2}\right) \right)}.
\end{align*}
Finally, with $p = 1-\mu(A)$ and $t = \theta \log (1+\lambda^{(k)}\epsilon^{2})$, we obtain
\[
  \prod_{i=1}^{k} {\left(1+\lambda^{(k)}\epsilon^{2}\right)}^{N_{i}} \leq {\left(\e^{-pt} + p \e^{(1-p)t}\right)}^{1/a_{(1)}}.
\]
We easily check that, the quantity in the right-hand side is minimal for $t = \log \frac{1}{1-p}$ at which it takes the value ${(1-p)}^{p-1} = {\mu(A)}^{-\mu(A)/a_{(1)}}$.
Thus,
\begin{equation}\label{equation:bound_product2}
  \prod_{i=1}^{k} {(1+\lambda^{(k)}\epsilon^{2})}^{N_{i}} \leq \frac{1}{{\mu(A)}^{\mu(A)/a_{(1)}}}.
\end{equation}
Combining~\cref{equation:bound_product1,equation:bound_product2} with~\cref{equation:bound_N0} and the same argument as for~\cref{eq:Psi}, we obtain the two announced bounds.
\end{proof}
From~\cref{proposition:alternative-result}, we can derive bounds on the $\lambda^{(k)}$'s.
The proof is the same as the one of~\cref{proposition:bound-eigenvalue} and is omitted.

\begin{proposition}\label{proposition:alternative-bound-eigenvalue}
  Let $(E,d,\mu)$ be a metric measured space and $\lambda^{(k)}$ be defined as in~\eqref{metric_eigenvalues}.
  Let $A_{1}, \dots, A_{k}$ be measurable sets, then, with $r = \frac{1}{2} \min_{i \ne j} d(A_{i},A_{j})$ and $A_{0} = E \setminus {(\cup A_{i})}_{r}$,
  \begin{align*}
    \lambda^{(k)} & \leq \frac{1}{r^{2}} \psi\left(\frac{1}{c} \ln \frac{a_{(1)}}{\mu(A_{0})} + \frac{1}{c} k \ln \frac{1}{a_{(1)}} \right);\\
    \lambda^{(k)} & \leq \frac{1}{r^{2}} \psi\left(\frac{1}{c}  \ln \frac{a_{(1)}}{\mu(A_{0})} + \frac{1}{c} \frac{\mu(A)}{a_{(1)}} \ln \frac{1}{\mu(A)} \right),
  \end{align*}
  where $\psi(x) = \max(x,x^{2})$ and $a_{(1)} = \min_{1 \leq i \leq k} \mu(A_{i})$.
\end{proposition}

\subsection{Comparison with the result of Chung-Grigor'yan-Yau}\label{sec:comparison}
In~\cite{ChungGrigoryanYau1996}, the authors obtained the following result: 
\begin{theorem}[Chung-Grigoryan-Yau~\cite{ChungGrigoryanYau1996}]\label{thm:CGY}
Let $M$ be a compact connected smooth Riemannian manifold equipped with its geodesic distance $d$ and normalized Riemannian volume $\mu$.
For any $k\geq 1$ and any family of sets $A_0,\ldots,A_k$, it holds
\begin{equation}\label{eq:CGY}
\lambda^{(k)} \leq \frac{1}{\min_{i \neq j} d^2(A_{i},A_{j})} \max_{i \neq j} \log {(\frac{4}{\mu(A_i)\mu(A_j)})}^{2},
\end{equation}
where $1=\lambda^{(0)}\leq \lambda^{(1)}\leq\cdots \lambda^{(k)}\leq \cdots$ denotes the discrete spectrum of $-\Delta$.
\end{theorem}
Let us translate this result in terms of concentration of measure.
Let $A_1,\ldots,A_k$ be sets such that $r=\frac{1}{2}\min_{1\leq i<j\leq k} d(A_i,A_j)>0$ and define $A =A_1\cup\cdots \cup A_k$ and $A_0 = M\setminus A_s$, for some $0<s\leq r$.
Then, applying~\eqref{eq:CGY} to this family of $k+1$ sets gives the following inequality
\begin{equation}\label{eq:CGYbis}
\min\left(a_{(2)} ; 1-\mu(A_s)\right) \leq \frac{4}{a_{(1)}} \exp(-\sqrt{\lambda^{(k)}}s),
\end{equation}
with $a_{(1)}$ and $a_{(2)}$ being respectively the smallest number and the second smallest number among $(\mu(A_1),\ldots,\mu(A_k))$ (counted with multiplicity).
Note that the right hand side is less than or equal to $a_{(2)}$ if and only if $s \geq s_o := \frac{1}{\sqrt{\lambda^{(k)}}} \log\left(\frac{4}{a_{(1)}a_{(2)}}\right)$, so that~\eqref{eq:CGYbis} is equivalent to the following statement:
\begin{equation}\label{eq:CGYter}
\mu(A_s) \geq 1-\frac{4}{a_{(1)}} \exp(-\sqrt{\lambda^{(k)}}s), \qquad \forall s \in [\min(s_o,r) ; r].
\end{equation}
We note that~\cref{eq:CGYter} holds for any family of sets, whereas the inequality given in~\cref{main-result} is only true when $(\mu(A_1),\ldots,\mu(A_k)) \in \Delta_k$.
Also due to the fact that the constant $c$ appearing in~\cref{main-result} is less than $1$,~\cref{eq:CGYter} is asymptotically better than ours (see also~\cref{rem:c} above).
On the other hand, one sees that~\cref{eq:CGYter} is only valid for $s$ large enough (and its domain of validity can thus be empty when $s_{o} > r$) whereas our inequality is true on the whole interval $(0,r]$.
It does not seem also possible to iterate~\cref{eq:CGYter} as we did in~\cref{cor:iteration}.
Finally, observe that the method of proof used in~\cite{ChungGrigoryanYau1996} and~\cite{ChungGrigoryanYau1997} is based on heat kernel bounds and is very different from ours.

Let us translate~\cref{thm:CGY} in a form closer to our~\cref{proposition:bound-eigenvalue}.
Fix $k$ sets $A_{1}, \dots, A_{k}$ such that $(\mu(A_{1}),\dots, \mu(A_{k})) \in \Delta_{k}$.
Let $2r = \min d(A_{i},A_{j})$, where the infimum runs on $i,j = 1,\dots,k$ with $i \ne j$.
We have to choose a $(k+1)$-th set.
In view of~\cref{thm:CGY}, the most optimal choice is to choose $A_{0} = E \setminus {(\cup A_{i})}_{r}$.
Indeed, it is the biggest set (in the sense of inclusion) such that $\min d(A_{i},A_{j}) = r$ where this time the infimum runs on $i,j = 0, \dots, k$ and $i \ne j$.
We let $a_{(0)} = \mu(A_{0})$ and we remark that if $(\mu(A_{1}), \dots, \mu(A_{k})) \in \Delta_{k}$ then $a_{(0)} \leq a_{(1)}$.
The bound~\cref{eq:CGY} can be read: for all $r > 0$,
\[
  \lambda^{(k)} \leq \frac{1}{r^{2}} {\left(\log \frac{4}{a_{(1)} a_{(0)}}\right)}^{2}.
\]
Therefore, to compare it to our bound, we need to solve
\[
  \phi^{-1}{\left(\frac{1}{c} \log \frac{a_{(1)}}{a_{(0)}}\right)}^{2} \leq {\left(\log \frac{4}{a_{(1)} a_{(0)}}\right)}^{2}.
\]
Because the right-hand side is always $\geq 1$, taking the square root and composing with the non-decreasing function $\phi$ yields
\[
  \frac{1}{c} \log \frac{a_{(1)}}{a_{(0)}} \leq \log \frac{4}{a_{(1)}a_{(0)}}.
\]
That is
\[
  a_{(1)}^{1+c} \leq 4^{c} a_{(0)}^{1-c}.
\]
In other words, on some range our bound is better and in some other range their bound is better.
However, if the constant $c = 1$ could be attained in~\cref{main-result}, this would show that our bound is always better.
Note that comparing the bounds obtained in~\cref{proposition:alternative-bound-eigenvalue} and the one of~\cite{ChungGrigoryanYau1996} is not so clear as, without the assumption that $(\mu(A_{1}),\dots,\mu(A_{k})) \in \Delta_{k}$ it is not necessary that $a_{(0)} \leq a_{(1)}$ and in that case we would have to compare different sets.

\section{Eigenvalue estimates for non-negatively curved spaces}\label{sec:examples}
We recall the values of the $\lambda^{(k)}$'s that appear in~\cref{main-result} in the case of two important models of positively curved spaces in geometry.
Namely:
\begin{enumerate}[(i)]
  \item The $n$-dimensional sphere of radius $\sqrt{\frac{n-1}{\rho}}$, $\dS^{n,\rho}$ endowed with the natural geodesic distance $d_{n,\rho}$ arising from its canonical Riemannian metric and its normalized volume measure $\mu_{n,\rho}$ which has constant Ricci curvature equals to $\rho$ and dimension $n$.
  \item The $n$-dimensional Euclidean space $\dR^{n}$ endowed with the $n$-dimensional Gaussian measure of covariance $\rho^{-1} \mathrm{Id}$,
    \[
    \gamma_{n,\rho}(\dd x) = \frac{\rho^{n/2} \e^{-\rho|x|^{2}/2}}{{(2 \pi)}^{n/2}} \dd x.
  \]
    This space has dimension $\infty$ and curvature bounded below by $\rho$ in the sense of~\cite{BakryEmeryDiffusionsHypercontractives}.
\end{enumerate}
These models arise as weighted Riemannian manifolds without boundary having a purely discrete spectrum.
In that case, it was proved in~\cite[Proposition 3.2]{Milman2015} that the $\lambda_{k}$'s of~\cref{metric_eigenvalues} are exactly the eigenvalues (counted with multiplicity) of a self-adjoint operator that we give explicitly in the following.
Using a comparison between eigenvalues of~\cite{Milman2015}, we obtain an estimates for eigenvalues in the case of log-concave probability measure over the Euclidean $\dR^{n}$.

\begin{example}[Spheres]
  On $\dS^{n,\rho}$, the eigenvalues of minus the Laplace-Beltrami operator (see for instance~\cite[Chapter 3]{AtkinsonHan}) are of the form $\rho^{-2} {(n-1)}^{2} l (l+n-1)$ for $l \in \dN$ and the dimension of the corresponding eigenspace $H_{l,n}$ is
  \[
    \dim H_{l,n} = \frac{2l + n -1}{l} \binom{l+n-2}{l-1},\ \text{if} \ l > 0 \ \and \ \dim H_{l,n} = 1, \ \text{if} \ l = 0.
  \]
  Consequently,
  \[
    D_{l,n} := \dim \bigoplus_{l'=0}^{l} H_{l',n} = \binom{n+l}{l} + \binom{n+l-1}{l-1},
  \]
  and $\lambda^{(k)} = \rho^{-2} {(n-1)}^{2} l (l+n-1)$ if and only if $D_{l-1,n} < k \leq D_{l,n}$ where $\lambda^{(k)}$ is the $k$-th eigenvalues of $- \Delta_{\dS^{n,\rho}}$ and coincides with the variational definition given in~\eqref{metric_eigenvalues}.
\end{example}

\begin{example}[Gaussian spaces]
  On the Euclidean space $\dR^{n}$, equipped with the Gaussian measure $\gamma_{n,\rho}$, the corresponding weighted Laplacian is $\Delta_{\gamma_{n,\rho}} = \Delta_{\dR^{n}} - \rho x \cdot \nabla$.
  The eigenvalues of $-\Delta_{\gamma_{n,\rho}}$ are exactly of the form $\rho^{2} q$ and the dimension of the associated eigenspace $H_{q,n}$ is
  \[
    \dim H_{q,n} = \binom{n+q-1}{q}.
  \]
  Consequently,
  \[
    D_{q,n} := \dim \bigoplus_{q'=0}^{q} H_{q',n} = \binom{n+q}{q},
  \]
  and $\lambda^{(k)} = \rho^{-2} q$ if and only if $D_{q-1,n} < k \leq D_{q,n}$ where $\lambda^{(k)}$ is the $k$-th eigenvalues of $- \Delta_{\gamma_{n,\rho}}$ and coincides with the variational definition given in~\eqref{metric_eigenvalues}.
\end{example}

\begin{example}[Log-concave Euclidean spaces]
We study the case where $E = \dR^{n}$, $d$ is the Euclidean distance and $\mu$ is a strictly log-concave probability measure.
By this we mean that $\mu(\dd x) = \e^{-V(x)} \dd x$, where $V \colon \dR^{n} \to \dR$ such that $V$ is $\cC^{2}$ and satisfying $\nabla^{2} V \geq K$ for some $K > 0$.
It is a consequence of~\cite[Proposition 4]{BakryEmeryDiffusionsHypercontractives} that such a condition on $V$ implies that the semigroup generated by the solution of the stochastic differential equation
\[
  \dd X_{t} = \sqrt{2} \dd B_{t} - \nabla V(X_{t}) \dd t,
\]
where $B$ is a Brownian motion on $\dR^{n}$, satisfies the curvature-dimension $CD(\infty,K)$ of Bakry-Emery and, therefore, holds the log-Sobolev inequality, for all $f \in \cC_{c}^{\infty}(\dR^{n})$,
\[
  \Ent_{\mu} f^{2} \leq \frac{2}{K} \int |\nabla f(x)|^{2} \mu(\dd x).
\]
Such an inequality implies the super-Poincaré of~\cite[Theorem 2.1]{WangSuperPoincare} that in turns implies that the self-adjoint operator $L = - \Delta + \nabla V \cdot \nabla$ has a purely discrete spectrum.
In that case, the $\lambda^{(k)}$ of~\eqref{metric_eigenvalues} corresponds to these eigenvalues and~\cite{Milman2015} showed that
\[
  \lambda^{(k)} \geq \lambda^{(k)}_{\gamma_{n,\rho}},
\]
where $\lambda^{(k)}_{\gamma_{n,\rho}}$ is the eigenvalues of $-\Delta_{\gamma_{n,\rho}}$ of the previous example.
\end{example}

\section{Extension to Markov chains}\label{section:Markov}
As in the classical case (see~\cite[Theorem 3.3]{Ledoux2001}), our continuous result admits a generalization on finite graphs or more broadly in the setting of Markov chains on a finite state space.
We consider a finite set $E$ and $X = (X_{n})_{n\in \dN}$ be a irreducible time-homogeneous Markov chain with state space $E$.
We write $p(x,y) = \dP(X_{1} = y | X_{0} = x)$ and we regard $p$ as a matrix.
We assume that $p$ admits a reversible probability measure $\mu$ on $E$ :  $p(x,y) \mu(x) = p(y,x) \mu(y)$ for all $x,y \in E$ (which implies in particular that $\mu$ is invariant).
The Markov kernel $p$ induces a graph structure on $E$ by the following procedure.
Set the elements of $E$ as the vertex of the graph and for $x,y \in E$ connect them with an edge if $p(x,y) > 0$.
As the chain is irreducible, this graph is connected.
We equip $E$ with the induced graph distance $d$.
We write $L = p - I$, where $I$ stands for the identity matrix.
The operator $-L$ is a symmetric positive operator on $\cL^{2}(\mu)$.
We let $\lambda^{(k)}$ be the eigenvalues of this operator.
Then, our~\cref{main-result} extends as follows:
\begin{theorem}\label{result-markov-chains}
For any $k\geq 1$ and for all sets $A_1,\ldots,A_k \subset E$ such that $\min_{i\neq j} d(A_i,A_j) \geq 1$ and $(\mu(A_1),\ldots,\mu(A_k)) \in \Delta_k$ the set $B=A_1\cup A_2\cup \cdots \cup A_k$ satisfies
\[
  \mu(B_{n}) \geq 1 - (1-\mu(B)) {\left(1 + \lambda^{(k)}\right)}^{-n},
\]
for all $1 \leq n \leq \frac{1}{2}\min_{i\neq j} d(A_i,A_j)$ where $\lambda^{(k)}$ is the $k$-th eigenvalue of the operator $-L$ acting on $\cL^{2}(\mu)$.
\end{theorem}

\begin{proof}
  We let $\Pi(x,y) = p(x,y) \mu(x)$ and
  \[
    \CE(f,g) = \frac{1}{2} \sum (f(y)-f(x)) (g(y)-g(x)) \Pi(x,y) = \langle f, -Lg \rangle_{\mu}.
\]
  For any set $A$, we define the discrete boundary of $A$ as $\partial A = A_{1} \setminus A \cup {(A^{C})}_{1} \setminus A^{C}$.
  Let $(X_{n})$ be the Markov chain with transition kernel $p$ and initial distribution $\mu$.
  By reversibility of $\mu$, $(X_{0}, X_{1})$ is an exchangeable pair of law $\Pi$ whose the marginals are given by $\mu$.
  Then, for a set $U$, we have
  \[
    \CE(1_{U}) = \dE 1_{U}(X_{0}) (1_{U}(X_{0}) - 1_{U}(X_{1})) = \dP(X_{0} \in U, X_{1} \not \in U) \leq \dP(X_{1} \in \partial U) = \mu(\partial U).
  \]
Observe that if $d(U,V) \geq 1$, $U$ and $V$ are disjoint and $U \times V \not \in \supp \Pi$ so that $\CE(1_{U}, 1_{V}) = 0$.
By Courant-Fischer's min-max theorem
\[
  \lambda^{(k)} = \min_{\dim V = k+1} \max_{f \in V} \frac{\CE(f,f)}{\mu(f^{2})}.
\]
Choose sets $A_{1}, \dots, A_{k}$ with $d(A_{i}, A_{j}) \geq 2n$ ($i \neq j$) and $(\mu(A_{1}), \dots, \mu(A_{k})) \in \Delta_{k}$.
  Set $f_i = 1_{A_{i}}$.
  The $f_{i}$'s have disjoint support and so they are orthogonal in $L^{2}(\mu)$.
  By the previous variational representation of $\lambda^{(k)}$, we have
\[
  \lambda^{(k)} \leq%
  \sup_{a_i} \frac%
  {  \CE \left(\sum_{i=0}^{k} a_{i} f_{i} \right) }%
  { \int {\left(\sum_{i=0}^{k} a_{i} f_{i} \right)}^2 \dd \mu}%
  =%
  \sup_{a_{i}} \frac{\sum a_{i} a_{i'} \CE(f_{i}, f_{i'})}{\sum a_{i} a_{i'} \int f_{i} f_{i'} \dd \mu}
  =%
  \sup_{a_i} \frac%
  { \sum_{i=0}^{k} a^{2}_{i} \CE(f_{i}) }%
  { \sum_{i=0}^{k} a_{i} \int f_{i}^2 \dd \mu}.
  \]
  In other words,
  \[
    \lambda^{(k)} \leq \max_{i=0, \dots, k} \frac{\mu({(A_{i})}_{1}) +\mu({(A_{i}^{C})}_{1}) - 1}{\mu(A_{i})} \leq \frac{\mu({(A_{i})}_{1}) - \mu(A_{i})}{\mu(A_{i})},
  \]
  where the last inequality comes from the fact that, by~\cref{lemma:inclusion}, $\mu(E \setminus {(E \setminus A)}_{1}) \geq \mu(A)$.
  Consider the set $B = \cup_{i=1}^{k} A_{i}$ and choose $A_{0} = E \setminus B_{1}$.
  In that case, by~\cref{lemma:max_of_quotient} with $\epsilon = 1$, we have
  \[
    \max_{i=0, \dots, k} \frac{\mu({(A_{i})}_{1})}{\mu(A_{i})} \leq \frac{1 - \mu(B)}{1 - \mu(B_{1})}.
  \]
  Thus, we proved that
  \[
    (1 + \lambda^{(k)}) (1 - \mu(B_{1})) \leq (1 - \mu(B)).
  \]
  We derive the announced result by an immediate recursion.
\end{proof}

\section{Functional forms of the multiple sets concentration property}\label{sec:functional}
We investigate the functional form of the multi-sets concentration of measure phenomenon results obtained in~\cref{section:concentration,section:Markov}.

\begin{proposition}\label{prop:conc-equiv}
Let $(E,d)$ be a metric space equipped with a Borel probability measure $\mu$.
Let $\alpha_k : [0,\infty) \to [0,\infty)$.
The following properties are equivalent:
\begin{enumerate}
\item For all Borel sets $A_1,\ldots,A_k \subset E$ such that $(\mu(A_1),\ldots,\mu(A_k)) \in \Delta_k$, the set $A = A_1\cup \cdots \cup A_k$ satisfies 
\begin{equation}\label{eq:conc-alpha}
\mu(A_r) \geq 1- (1-\mu(A)) \alpha_k (r),\qquad \forall 0<r \leq \frac{1}{2}\min_{i\neq j} d(A_i,A_j).
\end{equation}
\item For all $1$-Lipschitz functions $f_1,\ldots,f_k : E \to \mathbb{R}$ such that the sublevel sets $A_i=\{f_i\leq 0\}$ are such that $(\mu(A_1),\ldots,\mu(A_k)) \in \Delta_k$, the function $f^* = \min (f_1,\ldots, f_k)$ satisfies
\[
\mu(f^* < r) \geq 1- \mu(f^*\leq 0) \alpha_k (r),\qquad \forall 0<r \leq \frac{1}{2}\min_{i\neq j} d(A_i,A_j).
\]  
\end{enumerate}
\end{proposition}
Together with~\cref{main-result} or~\cref{result-markov-chains}, one thus sees that the presence of multiple wells can improve the concentration properties of a Lipschitz function.

\begin{proof}
It is clear that (2) implies (1) when applied to $f_i(x) = d(x,A_i)$, in which case $A_i=\{f_i\leq 0\}$ and $f^*(x) = d(x,A)$.
The converse is also very classical.
First, observe that $\{f^*<r\} = \cup_{i=1}^k\{f_i<r\}$.
Then, since $f_i$ is $1$-Lipschitz, it holds $A_{i,r} \subset \{f_i <r\}$ with $A_i=\{f_i\leq 0\}$ and so letting $A=A_1\cup\cdots \cup A_k$, it holds $A_r\subset \{f^*<r\}$.
Therefore, applying (1) to this set $A$ gives (2).
\end{proof}
\noindent
When~\cref{eq:conc-alpha} holds, we will say that the probability metric space $(E,d,\mu)$ satisfies the multi-set concentration of measure property of order $k$ with the concentration profile $\alpha_k$.

In the usual setting ($k=1$), the concentration of measure phenomenon implies deviation inequalities for Lipschitz functions around their median.
The next result generalizes this well known fact to $k>1$.

\begin{proposition}\label{prop:conc-quantile} Let $(E,d,\mu)$ be a probability metric space satisfying the multi-set concentration of measure property of order $k$ with the concentration profile $\alpha_k$ and $f:E\to \mathbb{R}$ be a $1$-Lipschitz function.
If $I_1,\ldots, I_k\subset \mathbb{R}$ are $k$ disjoint Borel sets such that $(\mu(f\in I_1),\ldots,\mu(f\in I_k)) \in \Delta_k$, then it holds
\[
\mu \left(f \in \cup_{i=1}^k I_{i,r}\right) \geq 1 - (1- \mu(f \in \cup_{i=1}^k I_{i})) \alpha_k(r),\qquad \forall 0<r \leq \frac{1}{2}\min_{i\neq j} d(I_i,I_j)
\]   
\end{proposition}
\begin{proof}
Let $\nu$ be the image of $\mu$ under the map $f$.
Since $f$ is $1$-Lipschitz, the metric space $(\mathbb{R}, |\,\cdot\,|, \nu)$ satisfies the multi-set concentration of measure property of order $k$ with the same concentration profile $\alpha_k$ as $\mu$.
Details are left to the reader.
\end{proof}
Let us conclude this section by detailing an application of potential interest in approximation theory. 

Suppose that $f : E \to \dR$ is some $1$-Lipschitz function and $A_1,\ldots ,A_k$ are (pairwise disjoint) subsets of $E$ such that $(\mu(A_1),\ldots, \mu(A_k)) \in \Delta_k$. Let us assume that the restrictions $f_{|A_i}$, $i\in \{1,\ldots, k\}$ are known and that one wishes to estimate or reconstruct $f$ outside $A= \cup_{i=1}^k A_i$. To that aim, one can consider an explicit $1$-Lipschitz extension of $f_{|A}$, that is to say a $1$-Lipschitz function $g:E \to \dR$ (constructed based on our knowledge of $f$ on $A$ exclusively) such that $f=g$ on $A$. There are several canonical ways to perform the extension of a Lipschitz function defined on a sub domain (known as Kirszbraun-McShane-Whitney extensions \cite{Kirszbraun1934,McShane1934,Whitney1934}). One can consider for instance the functions
\[
g_+(x) = \inf_{y\in A}\{f(y)+d(x,y)\}\qquad \text{or}\qquad g_{-}(x) = \sup_{y\in A}\{ f(y)-d(x,y)\},\qquad x\in E.
\]
It is a very classical fact that functions $g_-$ and $g_+$ are $1$-Lipschitz extensions of $f_{|A}$ and moreover that any extension $g$ of $f_{|A}$ satisfies $g_{-}\leq g\leq g_+$ (see \textit{e.g} \cite{Heinonen2005}).

The following simple result shows that, for any $1$-Lipschitz extension $g$ of $f_{|A}$, the probability of error $\mu(|f-g|>r)$ is controlled by the multi-set concentration profile $\alpha_k$. In particular, in the framework of our \cref{main-result}, this probability of error is expressed in terms of $\lambda^{(k)}$. 

\begin{proposition}\label{prop:conc-extension}Let $(E,d,\mu)$ be a probability metric space satisfying the multi-set concentration of measure property of order $k$ with the concentration profile $\alpha_k$ and $f:E\to \mathbb{R}$ be a $1$-Lipschitz function. Let $A_1,\ldots A_k$ be subsets of $E$ such that $(\mu(A_1),\ldots, \mu(A_k)) \in \Delta_k$ ; then for any $1$-Lipschitz extension $g$ of $f_{|A}$, it holds
\[
\mu(|f-g|\geq r) \leq (1-\mu(A)) \alpha_k(r/2),\qquad \forall 0<r\leq \min_{i\neq j} d(A_i,A_j).
\]
\end{proposition}
\proof
The function $h:E\to \dR$ defined by $h(x) = |f-g|(x)$, $x \in E$,  is $2$-Lipschitz and vanishes on $A$. Therefore, for any $x \in E $ and $y \in A$, it holds $h(x) \leq h(y)+ 2d(x,y) = 2d(x,y)$. Optimizing over $y\in A$ gives that $h(x) \leq 2 d(x,A)$. Therefore $\{ h \geq r\} \subset \{x : d(x,A) \geq r/2\} = \left(A_{r/2}\right)^c$ and so, if $0<r\leq \min_{i\neq j} d(A_i,A_j)$, it holds
\[
\mu(|f-g|\geq r) \leq (1-\mu(A)) \alpha_k(r/2).
\]
\endproof

\begin{remark}
Let us remark that \cref{prop:conc-equiv,prop:conc-quantile,prop:conc-extension} can be immediately extended under the following more general (but notationally heavier) multi-set concentration of measure assumption: there exists functions $\alpha_k:[0,\infty)\to [0,\infty) $ and $\beta_k: [0,\infty)^k \to [0,\infty]$ such that for all Borel sets $A_1,\ldots,A_k \subset E$, the set $A = A_1 \cup\cdots\cup A_k$ satisfies
\[
\mu(A_r) \geq 1 - \beta_k(\mu(A_1),\cdots,\mu(A_k)) \alpha_k(r),\qquad \forall 0<r\leq \frac{1}{2}\min_{i\neq j} d(A_i,A_j).
\] 
This framework contains the preceding one, by choosing $\beta_k(a)=1-\sum_{i=1}^ka_i$ if $a=(a_1,\ldots,a_k) \in \Delta_k$ and $+\infty$ otherwise. It also contains the concentration bounds obtained in \cref{proposition:alternative-result}, corresponding respectively to 
\[
\beta_k(a) = \frac{1-\sum_{i=1}^k a_i}{\prod_{i=1}^k a_i}, \text{ and }\beta_k(a) = \frac{1-\sum_{i=1}^k a_i}{\left(\sum_{i=1}^k a_i\right)^{\sum_{i=1}^k a_i / \min(a_1,\cdots,a_k)}},\quad
a = (a_1,\ldots,a_k).
\] 
\end{remark}

\section{Open questions}
\label{section:questions}
We list open questions related to the multi-set concentration of measure phenomenon.
\subsection{Gaussian multi-set concentration}
Using the terminology introduced in~\cref{sec:functional},~\cref{main-result} and the material exposed in~\cref{sec:examples} tell us that, if $\mu$ has a density of the form $e^{- V}$ with respect to Lebesgue measure on $\dR^n$ with a smooth function $V$ such that $\mathrm{Hess}\,V \geq \rho >0$, then the probability metric space $(\dR^n, |\,\cdot\,|,\mu)$ satisfies the multi-set concentration of measure property of order $k$ with the concentration profile
\[
\alpha_k(r) = \exp\left(- c \min(r^2 \lambda^{(k)}_{\gamma_n,\rho} ; r \sqrt{\lambda^{(k)}_{\gamma_n,\rho}})\right),\qquad r\geq0,
\]
where $\lambda^{(k)}_{\gamma_n,\rho}$ denotes the $k$th eigenvalue of the $n$-dimensional centered Gaussian measure with covariance matrix $\rho^{-1}\mathrm{Id}$.
Since the measure $\mu$ satisfies the log-Sobolev inequality, it is well known that it satisfies a (classical) Gaussian concentration of measure inequality.
Therefore, it is natural to conjecture that $\mu$ satisfies a multi-set concentration of measure property of order $k\geq 1$ with a profile of the form
\[
\beta_k(r) = \exp\left(- C_{k,\rho,n} r^2 \right),\qquad r\geq0,
\] 
for some constant $C_{k,\rho,n}$ depending solely on its arguments.
In addition, it would be interesting to see how usual functional inequalities (Log-Sobolev, transport-entropy, \ldots) can be modified to catch such a concentration of measure phenomenon.

\subsection{Equivalence between multi-set concentration and lower bounds on eigenvalues in non-negative curvature}
Let us quickly recall the main finding of E. Milman~\cite{Milman2009, Milman2010}, that is, under non-negative curvature assumptions, a concentration of measure estimate implies a bound on the spectral gap.
Let $\mu$ be a probability measure with a density of the form $e^{-V}$ on a smooth connected Riemannian manifold $M$ with $V$ a smooth function such that 
\begin{equation}\label{eq:curvature}
\mathrm{Ric}+\mathrm{Hess}\,V\geq0.
\end{equation}
Assume that $\mu$ satisfies a concentration inequality of the form: for all $A \subset M$ such that $\mu(A)\geq1/2$
\[
\mu(A_r) \geq 1 - \alpha(r),\qquad r\geq0,  
\]   
where $\alpha$ is a function such that $\alpha(r_o)<1/2$ for at least one value $r_o>0$.
Then, letting $\lambda^{(1)}$ be the first non zero eigenvalue of the operator $- \Delta + \nabla V\cdot\nabla$, it holds $\lambda^{(1)} \geq \frac{1}{4}{\left(\frac{1-2\alpha(r_o)}{r_o}\right)}^{2}$.
It would be very interesting to extend Milman's result to a multi-set concentration setting.
More precisely, if $\mu$ satisfies the curvature condition~\eqref{eq:curvature} and the multi-set concentration of measure property of order $k$ with a profile of the form $\alpha_k(r) = \exp(- \min(a r^2,\sqrt{a}r))$, $r\geq0$, can we find a universal function $\varphi_{k}$ such that $\lambda^{(k)} \geq \varphi_k(a)$?

This question already received some attention in recent works by Funano and Shioya~\cite{Funano2017, FunanoShioya2013}.
In particular, let us mention the following improvement of the Chung-Grigor'yan-Yau inequality obtained in~\cite{Funano2017}.
There exists a universal constant $c>1$ such that if $\mu$ is a probability measure satisfying the non-negative curvature assumption~\eqref{eq:curvature}, it holds: for any family of sets $A_0,A_1,\ldots,A_l$ with $1\leq l\leq k$
\begin{equation}\label{eq:CGYimproved}
\lambda^{(k)} \leq c^{k-l+1}\frac{1}{\min_{i \neq j } d^2(A_{i},A_{j})} \max_{i \neq j} \log {(\frac{4}{\mu(A_i)\mu(A_j)})}^{2}.
\end{equation}
Note that the difference with~\eqref{eq:CGY} is that $\lambda^{(k)}$ is estimated by a reduced number of sets.
Using~\eqref{eq:CGYimproved} (with $l=1$) together with Milman's result recalled above, Funano showed that there exists some constant $C_k$ depending only on $k$ such that under the curvature condition~\eqref{eq:curvature}, it holds $\lambda^{(k)} \leq C_k \lambda^{(1)}$ (recovering the main result of~\cite{FunanoShioya2013}).
The constant $C_k$ is explicit (contrary to the constant of~\cite{FunanoShioya2013}) and grows exponentially when $k\to \infty$.
This result has been then improved by Liu~\cite{Liu2014}, where a constant $C_k = O(k^2)$ has been obtained.
As observed by Funano~\cite{Funano2017}, a positive answer to the open question stated above would yield that under~\eqref{eq:curvature} the ratios $\lambda^{(k+1)}/\lambda^{(k)}$ are bounded from above by a universal constant.

\bibliographystyle{plain}

\begin{thebibliography}{}

\end{thebibliography}


\begin{thebibliography}{10}

\bibitem{AidaStroock1994}
S.~Aida and D.~Stroock.
\newblock Moment estimates derived from {P}oincar\'e and logarithmic {S}obolev
  inequalities.
\newblock {\em Math. Res. Lett.}, 1(1):75--86, 1994.

\bibitem{AmbrosioGhezziSobolevBVMMS}
Luigi Ambrosio and Roberta Ghezzi.
\newblock Sobolev and bounded variation functions on metric measure spaces.
\newblock In {\em Geometry, analysis and dynamics on sub-{R}iemannian
  manifolds. {V}ol. {II}}, EMS Ser. Lect. Math., pages 211--273. Eur. Math.
  Soc., Z\"urich, 2016.

\bibitem{AtkinsonHan}
Kendall Atkinson and Weimin Han.
\newblock {\em Spherical harmonics and approximations on the unit sphere: an
  introduction}, volume 2044 of {\em Lecture Notes in Mathematics}.
\newblock Springer, Heidelberg, 2012.

\bibitem{BakryEmeryDiffusionsHypercontractives}
D.~Bakry and Michel \'Emery.
\newblock Diffusions hypercontractives.
\newblock In {\em S\'eminaire de probabilit\'es, {XIX}, 1983/84}, volume 1123
  of {\em Lecture Notes in Math.}, pages 177--206. Springer, Berlin, 1985.

\bibitem{BakryGentilLedoux2014}
Dominique Bakry, Ivan Gentil, and Michel Ledoux.
\newblock {\em Analysis and geometry of {M}arkov diffusion operators}, volume
  348 of {\em Grundlehren der Mathematischen Wissenschaften [Fundamental
  Principles of Mathematical Sciences]}.
\newblock Springer, Cham, 2014.

\bibitem{BobkovLedoux1997}
S.~Bobkov and M.~Ledoux.
\newblock Poincar\'e's inequalities and {T}alagrand's concentration phenomenon
  for the exponential distribution.
\newblock {\em Probab. Theory Related Fields}, 107(3):383--400, 1997.

\bibitem{Buser1982}
Peter Buser.
\newblock A note on the isoperimetric constant.
\newblock {\em Ann. Sci. \'Ecole Norm. Sup. (4)}, 15(2):213--230, 1982.

\bibitem{Chavel1984}
Isaac Chavel.
\newblock {\em Eigenvalues in {R}iemannian geometry}, volume 115 of {\em Pure
  and Applied Mathematics}.
\newblock Academic Press, Inc., Orlando, FL, 1984.
\newblock Including a chapter by Burton Randol, With an appendix by Jozef
  Dodziuk.

\bibitem{CheegerLipschitz}
J.~Cheeger.
\newblock Differentiability of {L}ipschitz functions on metric measure spaces.
\newblock {\em Geom. Funct. Anal.}, 9(3):428--517, 1999.

\bibitem{ChungGrigoryanYau1997}
F.~R.~K. Chung, A.~Grigor$\prime$yan, and S.-T. Yau.
\newblock Eigenvalues and diameters for manifolds and graphs.
\newblock In {\em Tsing {H}ua lectures on geometry \&\ analysis ({H}sinchu,
  1990--1991)}, pages 79--105. Int. Press, Cambridge, MA, 1997.

\bibitem{ChungGrigoryanYau1996}
F.~R.~K. Chung, A.~Grigor'yan, and S.-T. Yau.
\newblock Upper bounds for eigenvalues of the discrete and continuous {L}aplace
  operators.
\newblock {\em Adv. Math.}, 117(2):165--178, 1996.

\bibitem{FriedmanTillich2000}
Joel Friedman and Jean-Pierre Tillich.
\newblock Laplacian eigenvalues and distances between subsets of a manifold.
\newblock {\em J. Differential Geom.}, 56(2):285--299, 2000.

\bibitem{Funano2017}
Kei Funano.
\newblock Estimates of eigenvalues of the {L}aplacian by a reduced number of
  subsets.
\newblock {\em Israel J. Math.}, 217(1):413--433, 2017.

\bibitem{FunanoShioya2013}
Kei Funano and Takashi Shioya.
\newblock Concentration, {R}icci curvature, and eigenvalues of {L}aplacian.
\newblock {\em Geom. Funct. Anal.}, 23(3):888--936, 2013.

\bibitem{GozlanRobertoSamson2015}
Nathael Gozlan, Cyril Roberto, and Paul-Marie Samson.
\newblock From dimension free concentration to the {P}oincar\'e inequality.
\newblock {\em Calc. Var. Partial Differential Equations}, 52(3-4):899--925,
  2015.

\bibitem{GromovMilman1983}
M.~Gromov and V.~D.~and Milman.
\newblock A topological application of the isoperimetric inequality.
\newblock {\em Amer. J. Math.}, 105(4):843--854, 1983.

\bibitem{Heinonen2005}
Juha Heinonen.
\newblock {\em Lectures on {L}ipschitz analysis}, volume 100 of {\em Report.
  University of Jyv\"askyl\"a Department of Mathematics and Statistics}.
\newblock University of Jyv\"askyl\"a, Jyv\"askyl\"a, 2005.

\bibitem{Kirszbraun1934}
M.~Kirszbraun.
\newblock Uber die zusammenziehende und lipschitzsche transformationen.
\newblock {\em Fundamenta Mathematicae}, 22:77--108, 1934.

\bibitem{Ledoux2001}
Michel Ledoux.
\newblock {\em The concentration of measure phenomenon}, volume~89 of {\em
  Mathematical Surveys and Monographs}.
\newblock American Mathematical Society, Providence, RI, 2001.

\bibitem{Lichnerowicz1958}
Andr\'e Lichnerowicz.
\newblock {\em G\'eom\'etrie des groupes de transformations}.
\newblock Travaux et Recherches Math\'ematiques, III. Dunod, Paris, 1958.

\bibitem{Liu2014}
S.~{Liu}.
\newblock {An optimal dimension-free upper bound for eigenvalue ratios}.
\newblock {\em ArXiv e-prints}, May 2014.

\bibitem{McShane1934}
E.~J. McShane.
\newblock Extension of range of functions.
\newblock {\em Bull. Amer. Math. Soc.}, 40(12):837--842, 1934.

\bibitem{Milman2015}
E.~{Milman}.
\newblock {Spectral Estimates, Contractions and Hypercontractivity}.
\newblock {\em ArXiv e-prints}, August 2015.

\bibitem{Milman2009}
Emanuel Milman.
\newblock On the role of convexity in isoperimetry, spectral gap and
  concentration.
\newblock {\em Invent. Math.}, 177(1):1--43, 2009.

\bibitem{Milman2010}
Emanuel Milman.
\newblock Isoperimetric and concentration inequalities: equivalence under
  curvature lower bound.
\newblock {\em Duke Math. J.}, 154(2):207--239, 2010.

\bibitem{Schmuckenschlager1998}
Michael Schmuckenschl\"ager.
\newblock Martingales, {P}oincar\'e type inequalities, and deviation
  inequalities.
\newblock {\em J. Funct. Anal.}, 155(2):303--323, 1998.

\bibitem{WangSuperPoincare}
Feng-Yu Wang.
\newblock Functional inequalities for empty essential spectrum.
\newblock {\em J. Funct. Anal.}, 170(1):219--245, 2000.

\bibitem{Whitney1934}
Hassler Whitney.
\newblock Analytic extensions of differentiable functions defined in closed
  sets.
\newblock {\em Trans. Amer. Math. Soc.}, 36(1):63--89, 1934.

\end{thebibliography}

\end{document}